\documentclass[10pt]{article}
\usepackage[english]{babel}

\usepackage{graphicx}
\usepackage{amsmath}
\usepackage{amssymb} 
\usepackage{amsthm}
\newcommand{\Isom}{\mathit{Isom}}
\newcommand{\N}{\mathbb{N}}

\newcommand{\C}{\mathbb{C}}
\newcommand{\R}{\mathbb{R}}

\newcommand{\ccl}[3]{#1^{#2}(#3)}
\newcommand{\corchete}[1]{\left\{{#1}\right\}}
\DeclareMathOperator{\mdeg}{mdeg}

\DeclareMathOperator{\End}{End}
\DeclareMathOperator{\rank}{rank}
\DeclareMathOperator{\Tr}{Tr}
\DeclareMathOperator{\Char}{char}
\DeclareMathOperator{\Nil}{\mathcal{N}}
\DeclareMathOperator{\Wedge}{\mathsf{\Lambda}}

\usepackage{color}
\usepackage[colorlinks=true]{hyperref}
\usepackage{fancyhdr}
\usepackage[left=3.95cm,top=3.5cm,right=3.95cm,bottom=2.5cm]{geometry}

\newtheorem{teo}{Theorem}[section]
\newtheorem{prop}[teo]{Proposition}
\newtheorem{lema}[teo]{Lemma}
\newtheorem{coro}[teo]{Corollary}
\theoremstyle{remark}\newtheorem{obs}[teo]{Remark}
\theoremstyle{definition}\newtheorem{ej}[teo]{Example}

\begin{document}

\title{\textbf{A new inequality about matrix products and a Berger-Wang formula}}

\author{\small{Eduardo Oreg\'on-Reyes}}
\author{
\small{EDUARDO OREG\'ON-REYES}}
%\\
%\small{July 13th, 1994}\\
%\small{Applying to PhD in Mathematics, Princeton University}}
\markboth{E Oreg\'on Reyes}{A new inequality about matrix products and a Berger-Wang formula}

\date{}
\maketitle

\begin{abstract}
We prove an inequality relating the norm of a product of matrices $A_n\cdots A_1$ with the spectral radii of  subproducts $A_j\cdots A_i$ with $1\leq i\leq j\leq n$. Among the consequences of this inequality, we obtain the classical Berger-Wang formula as an immediate corollary, and give an easier proof of a characterization of the upper Lyapunov exponent due to I. Morris. As main ingredient for the proof of this result, we prove that for a large enough $n$, the product $A_n\cdots A_1$ is zero under the hypothesis that $A_j\cdots A_i$ are nilpotent for all $1\leq i \leq j\leq n$.  
\end{abstract}

\section{Introduction}
Let $k$ be a field, and let $M_{d}(k)$ be the algebra of $d\times d$ matrices with coefficients in $k$. If $k=\R$ or $\C$, let $\|.\|$ be any norm on $k^d$, with the corresponding operator norm on $M_{d}(k)$ also denoted by $\|.\|$. The spectral radius of a matrix $A$ will be denoted by $\rho(A)$. Given a bounded set $\mathcal{M}\subset M_{d}(k)$, the \textit{joint spectral radius} of $\mathcal{M}$ is defined by the formula
\begin{equation}\label{jsr}
\mathfrak{R}(\mathcal{M})=\lim_{n\to \infty}{\left(\sup \corchete{\|A_{1}\cdots A_{n}\| : A_{i}\in \mathcal{M} }\right)^{1/n}}.
\end{equation}
By a submultiplicative argument, this quantity is well defined and finite, and the limit in the right hand side of \eqref{jsr} can be replaced by the infimum over $n$. 

The joint spectral radius was introduced by Rota and Strang \cite{rota}, and for a set $\mathcal{M}\subset M_{d}(k)$, represents the maximal exponential growth rate of the partial sequence of products $(A_1\cdots A_n)_{n}$ of a sequence of matrices $A_1,A_2,\dots$ with $A_i\in \mathcal{M}$. For this reason, this quantity has appeared in several mathematical contexts, making it an important object of study (see e.g. \cite{gur,jun,mathermorris,tsblo}). In particular, the question of whether the joint spectral radius may be approximated by periodic sequences plays an important role. The Berger-Wang formula gives a positive answer to this question in the case of bounded sets of matrices \cite{bw}:

\begin{teo}[Berger-Wang formula]\label{bergerwangantiguo}
If $\mathcal{M}\subset M_{d}(\C)$ is bounded, then 
\begin{equation}\label{formulabwantiguo}
\mathfrak{R}(\mathcal{M})=\limsup_{n\to \infty}{\left(\sup\corchete{\rho(A_{1}\cdots A_{n}) : A_{i}\in \mathcal{M}}\right)^{1/n}}.
\end{equation}
\end{teo}
This result has been generalized by Morris, to the context of linear cocycles (including infinite dimensional ones) \cite{morrisbw}, by using multiplicative ergodic theory. In the finite dimensional case, the problem of finding a formula similar to \eqref{formulabwantiguo}, when there is a Markov-type constraint on the allowed products was presented by Kozyakin \cite{kozyakin}. Although the result of Morris already applies to this kind of constraints, the novelty in Kozyakin's proof is that his arguments are purely linear algebraic, and are consequences of Theorem \ref{bergerwangantiguo}. 

Another tool to obtain results related to joint spectral radius was found by J. Bochi in \cite{bociineq}. In that work, he proved some inequalities that may be seen as lower bounds for joint spectral radii of sets of matrices in terms of the norms of such matrices. Following that method, the purpose of this article is to present an inequality relating the norm of the product of matrices with the spectral radii of subproducts. We will give an upper bound for the norm of the product of matrices $A_N\cdots A_1$ in terms of the spectral radii of its subproducts $A_{\beta}A_{\beta-1}\cdots A_{\alpha+1}A_{\alpha}$. This inequality will allow us to obtain relations similar to \eqref{formulabwantiguo}. It holds in an arbitrary local field where the notions of absolute value, norm, and spectral radius are well defined (see Section \ref{pruebadesigualdad} for a detailed explanation). Our main result is the following:

\begin{teo}\label{desigualdad} Let $d \in \N$, $k$ be a local field, and $\|.\|$ be a submultiplicative norm on $M_{d}(k)$. There exist constants $N = N(d)\leq  \prod^{d}_{i=1}{\binom{d}{i}}$, $r = r(d,N) \leq  (Nd+1)^{Nd^2+2}$, and $C = C(d,\|.\|) > 1$ such that for all $n\geq N$ and $A_1,\dots, A_n \in M_{d}(k)$:
\begin{equation}\label{eqdesigualdad}
\|A_n\cdots A_1\| \leq C \left(\prod_{1\leq i\leq n}{\|A_i\|}\right)\max_{1\leq \alpha\leq \beta\leq n}{ \left( \frac{\rho(A_{\beta}\cdots A_{\alpha}) }{ \prod_{\alpha\leq i\leq \beta}{\|A_i\|}} \right)^{1/r}},  
\end{equation}
where the right hand side is treated as zero if one of the $A_{i}$ is the zero matrix.
\end{teo}
So for large enough $n$, if the norm of the product $A_n\cdots A_1$ is comparable to (that is, not much smaller than) the product of the norms, then there exists a subproduct $A_{\beta}\cdots A_{\alpha}$ whose spectral radius is comparable to (that is, not much smaller than) $\prod_{\alpha\leq i \leq \beta}{\|A_{i}\|}$.

Note that inequality \eqref{eqdesigualdad} is homogeneous in each variable $A_{i}$. We will later show that the upper bound $N(d)\leq \prod^{d}_{i=1}{\binom{d}{i}}$ is not sharp, because $N(3)\leq 5$ (see Proposition \ref{casotres}). In addition, when $k=\C$, the constant $C$ in \eqref{eqdesigualdad} may be chosen independent of the norm $\|.\|$ and found explicitly, provided that $\|.\|$ is an operator norm (see Proposition \ref{independ} and Remark \ref{effective}).

The approach of using inequalities to prove results similar to \eqref{formulabwantiguo} was first used by Elsner \cite{elsner} in his proof of the Berger-Wang formula (with an inequality of a different nature from Bochi's work). Inequalities like \eqref{eqdesigualdad} also have been applied by I. Morris to study matrix pressure functions \cite{morrispr} and by the author in the context of isometries in Gromov hyperbolic spaces \cite{yo}. The novelty of the inequality presented here is that it respects the order in which the matrices are multiplied.  While previous works considered a sum or a maximum over \textit{all} possible subproducts of length $N$ with respect to a given alphabet of matrices, in Theorem \ref{desigualdad} we consider just one product of length $N$ together with its subproducts, hence our inequality does not follow from previously known Bochi-type results. In addition, our error in the upper bound in terms of spectral radii is multiplicative (the constant $C$) and not additive as in the case of Elsner's work. These distinctions allow inequality \eqref{eqdesigualdad} to be used in cases where only some specific kinds of products are allowed (see Theorem \ref{bergerwangnuevo} below), as well to relate asymptotic quantities (like the joint spectral radius) to non-asymptotic expressions, in a uniformly controlled way (see Theorem \ref{detodos}).

The proof of this inequality is based on the non trivial case of equality, where the right hand side of \eqref{eqdesigualdad} is zero but the matrices $A_i$ are non-zero. This occurs when $\rho(A_{j}\cdots A_{i})=0$ for all $1\leq i \leq j \leq N$, that is, when $A_{j}\cdots A_{i}$ are all nilpotent. Denote by $\Nil_{d}(k)$ the set of nilpotent elements of $M_{d}(k)$. Then define, for $n\geq 1$, the set $\Nil_{d}^{n}(k)$ of $n$-tuples $(A_1,\dots,A_n)\in {M_{d}(k)}^{n}$ such that $A_{j}\cdots A_{i}\in \Nil_{d}(k)$ for all $1\leq i \leq j \leq n$. 
The particular case of \eqref{eqdesigualdad} that we highlighted can be restated as follows:

\begin{teo}\label{nilpotentes} For all $d\geq 1$ there exists an integer $N=N(d)\geq 1$ such that, for every field $k$, if  $(A_1,\dots,A_N) \in \Nil_{d}^{N}(k)$, then the product $A_{N}\cdots A_{1}$ is zero.
\end{teo}
The proof of Theorem \ref{nilpotentes} is purely linear algebraic, exploiting the properties of the $n$-exterior power functor. This result may be compared with Levitzki's Theorem \cite[Thm.~2.1.7]{raro}, that asserts that for an algebraically closed field $k$, every semigroup  $S\subset M_{d}(k)$ of nilpotent matrices is simultaneously triangularizable. That is, there is some $B\in \mathrm{GL}_{d}(k)$ such that $BAB^{-1}$ is upper triangular with zero diagonal for every $A \in S$ (compare also with the Burnside-Schur Theorem for semigroups of matrices \cite{mcnaugh}). In particular, if $A_1,\dots,A_d\in S$, then the product $A_1 \cdots A_d$ is zero. As we show in Subsection \ref{computos}, the optimal $N(d)$ in Theorem \ref{nilpotentes} is in general larger than $d$, therefore the result presented here does not follow from Levitzki's Theorem nor Burnside-Schur Theorem, and we don't expect to obtain any information about the semigroup generated by $A_1,\dots,A_N$. In general, the matrices satisfying the hypothesis of Theorem \ref{nilpotentes} admit no normal form as simple as in Levitzki's Theorem.

\vspace{-4ex}\paragraph{Applications to Ergodic theory.} Let $(X,\mathcal{F},\mu)$ be a probability space, and let $T:X\rightarrow X$ be a measure preserving map. By a \textit{linear cocycle} over $X$, we mean a measurable map $A:X\rightarrow M_{d}(k)$ together with the family of maps $A^{n}$ defined by the formula
\begin{equation}
\ccl{A}{n}{x}=A(T^{n-1}x)\cdots A(Tx)A(x), \hspace{2mm} \text{ for }n\geq 1,x\in X.   \notag
\end{equation}
These maps satisfy the multiplicative cocycle relation $A^{m+n}(x)=A^{m}(T^{n}x)A^{n}(x)$ for all $m,n\geq 1,x \in X$.

We usually denote a linear cocycle by $\mathcal{A}=(X,T,A)$, and say that $\mathcal{A}$ is \textit{integrable} if $\max(\log{\|A\|},0)$ is integrable. In this case, Kingman's theorem implies that, for $\mu$-almost all $x\in X$, the limit $\lambda(x)=\lim_{n \to \infty}{\frac{\log{\|\ccl{A}{n}{x}\|}}{n}}\in [-\infty,\infty)$ exists, and moreover, $\lambda$ is $T$-invariant. This function is the upper \textit{Lyapunov exponent} of $A$, and is one of the most important concepts in multiplicative ergodic theory.

As an application of our inequality, we reprove the following theorem due to I. Morris (first tested numerically in \cite{fisicos} and proved by Avila-Bochi for $\mathrm{SL}_{2}(\R)$ in \cite[Thm.~15]{avibochi}).

\begin{teo}\label{bergerwangnuevo}\cite[Thm.~1.6]{morrisbw}
Let $T$ be a measure-preserving transformation of a probability space $(X,\mathcal{F},\mu)$ and let $A: X \rightarrow  M_{d}(k)$ be an integrable linear cocycle. If $\lambda$ is as before, then for $\mu$-almost all $x\in X$ we have
\begin{equation}\label{bww}
\limsup_{n\to \infty}{\frac{\log(\rho(\ccl{A}{n}{x}))}{n}=\lambda(x)}. 
\end{equation}
\end{teo}

While Morris's proof of this result relies on Oseledets Theorem, we will mainly use Theorem \ref{desigualdad} and a quantitative version of Poincar\'e's Recurrence Theorem.

\paragraph{Organization of the paper.} In Section \ref{secciondos} we prove Theorem \ref{nilpotentes} and compute $N(d)$ for $d=2,3$. Then in Section \ref{secciontres}, via Nullstellensatz we translate this theorem into a polynomial identity, from which we deduce Theorem \ref{desigualdad} in Section \ref{pruebadesigualdad}. We prove Theorem \ref{bergerwangnuevo} in Section \ref{seccioncinco}, and discuss some geometric consequences and analogies of this result in Section \ref{seccionseis}.

\section{Proof of Theorem \ref{nilpotentes}}\label{secciondos}

We begin the proof of Theorem \ref{nilpotentes} with some useful results. For a given vector space $V$ (over an arbitrary field), let $\End(V)$ be the algebra of linear endomorphisms of $V$. The dimension of the image of a linear transformation $T\in \End(V)$ will be denoted as $\rank(T)$. Also, let $\mathcal{N}^{n}(V)$ be the set of $n$-tuples $(T_1,\dots,T_n)\in \End(V)^{n}$ such that $T_{j}\cdots T_{i}$ is nilpotent for all $1\leq i \leq j \leq n$. With our previous notation, we have $\mathcal{N}^{n}(k^{d})=\mathcal{N}^{n}_{d}(k)$.

\begin{prop}\label{proprango}
Let $n\geq 1$ and $(T_1,\dots,T_n)\in \Nil^{n}(V)$ be such that $\rank(T_j)\leq 1$ for all $1\leq j\leq n-1$. If $v\in V$ and $T_n\cdots T_{1}v\neq 0$, then $v$, $T_1{v}$, $T_{2}T_{1}v$, $\dots,T_{n}\cdots T_{1}v$ are all distinct and form a linearly independent set. 
\end{prop}
\begin{proof} We will use induction on $n$. The case $n=1$ comes from the nilpotence of $T_1$. So, assume that the result holds for tuples in $\Nil^{n-1}(V)$ and let $(T_1,\dots,T_{n})\in \Nil^{n}(V)$ and $v\in V$ be as in the hypothesis.
Take a linear combination of $v$, $T_1{v}$, $\dots,T_{n}\cdots T_{1}$ of the form 
\begin{equation}\label{desigualdadpropo}
\lambda_{0}v+\lambda_{1}T_1{v}+\dots+\lambda_{n-1}T_{n-1}\cdots T_{1}v+\lambda_{n}T_{n}\cdots T_{1}v=0,
\end{equation}
and suppose that this linear combination is non trivial. As $(T_1,\dots,T_{n-1})\in \Nil^{n-1}(V)$ also satisfies the hypothesis with respect to $v$, by our inductive assumption we have $\lambda_{n}\neq 0$.
Now, apply $T_{n}\cdots T_{1}$ in \eqref{desigualdadpropo}. The rank condition over the maps $T_{j}$ and the fact that $(T_1,\dots,T_n)\in \Nil^{n}(V)$ imply that $(T_{j}\cdots T_{1})^{2}=0$, for all $1\leq j \leq n$. Hence, the left hand side of \eqref{desigualdadpropo} becomes $\lambda_{0}T_{n}\cdots T_{1}v$, forcing $\lambda_{0}=0$. But in that case, equation \eqref{desigualdadpropo} would be a non trivial linear combination of $\corchete{w,T_{2}{w},T_{3}T_{2}w,\dots,T_{n}\cdots T_{2}w}$, with $w=T_{1}v$. This is impossible by our inductive assumption, since $(T_{2},\dots,T_{n})\in \Nil^{n-1}(V)$ satisfies the hypothesis of the proposition with respect to $w$. We conclude that all linear combinations of $v$, $T_1{v}$, $T_{2}T_{1}v$, $\dots,T_{n}\cdots T_{1}v$ of the form \eqref{desigualdadpropo} are trivial, and hence this set is linearly independent with exactly $n+1$ elements. 
\end{proof}

\begin{coro}\label{cororango}
If $(T_1,\dots,T_d)\in \Nil^{d}(V)$ and $\rank(T_j)\leq 1$ for all $1 \leq j\leq d-1$, then $T_d\cdots T_1=0$. 
\end{coro}
\begin{proof}
Assume the contrary and let $v\in V$ be such that $T_{d}\cdots T_{1}v\neq 0$. Then by Proposition \ref{proprango}, the set $\corchete{v,T_1{v},T_{2}T_{1}v,\dots,T_{d}\cdots T_{1}v}$ would be a linearly independent set of cardinality greater than $\dim{V}$. A contradiction.
\end{proof}

For the next steps in our proof we need some fact about exterior powers. Recall that if $V$ is a vector space of dimension $d$, the \textit{r-fold exterior power} $\Wedge^{r}{V}$ is the vector space of alternating $r$-linear forms on the dual space $V^*$ (see e.g. \cite[XIX.1]{langalg}). Given a basis $\corchete{v_1,\dots, v_d}$ of $V$, the set 
$\corchete{v_{i_1}\wedge \cdots \wedge v_{i_r} : 1\leq i_1<\dots <i_r\leq d}$ is a basis of $\Wedge^{r}{V}$. Hence $\dim {\Wedge^{r}V}=\binom{d}{r}$. 

The exterior power also induces a map $\Wedge^{r}: \End(V) \rightarrow \End(\Wedge^{r}{V})$ given by the linear extension of $(\Wedge^{r}T)(w_1\wedge \cdots \wedge w_r)=(Tw_1\wedge \cdots \wedge Tw_r)$. This map is functorial: The relation $\Wedge^{r}(ST)=\Wedge^{r}(S)\Wedge^{r}(T)$ holds for all $S,T \in \End(V)$. This functor also induces a map $\Wedge^{r}: \Nil(V)\rightarrow \Nil(\Wedge^{r}{V})$ that extends to $\Nil^{n}(V)\rightarrow \Nil^{n}(\Wedge^{r}{V})$ for all $n\geq 1$.

Another important fact is that, when $T\in \Nil(V)$ and $\rank(T)=r>0$, then $\rank(\Wedge^{r}{T})=1$. This is because the image of $\Wedge^{r}{T}$ is generated by any $r$-form associated to the $r$-dimensional subspace $T(V)$. This remark is crucial in the end of our proof.

\begin{lema}\label{lemaa}
Let $1\leq r\leq d$ and $m=\binom{d}{r}$. Given $(T_1,\dots,T_m)\in \Nil^{m}(V)$, with $\rank(T_j) \leq r$ for all $1\leq j\leq m-1$, we have $\rank(T_{m}\cdots T_{1})<r$. 
\end{lema}
\begin{proof}
If that is the case then we will have $\rank(T_{j}T_{j-1}\cdots T_{i})\leq r$ for all $1\leq i\leq j\leq m-1$. Then the tuple $(\Wedge^{r}{T_1},\dots,\Wedge^{r}{T_m})\in \Nil^{m}(\Wedge^{r}{V})$ will satisfy the hypothesis of Corollary \ref{cororango}, and hence $\Wedge^{r}(T_m\cdots T_1)=0$, which implies that $\rank(T_m\cdots T_1 )<r$. 
\end{proof}

\begin{proof}[Proof of Theorem \ref{nilpotentes}]
Let $1\leq l< d$ and $r(l)= \binom{d}{1}\cdots \binom{d}{l}$. We claim that for all $(T_1,\dots,T_{r(l)})\in \Nil^{r(l)}(V)$ we have $\rank(T_{r(l)} \cdots T_{1})<d-l$. If so, the result follows with $N=r(d-1)=\binom{d}{1}\cdots \binom{d}{d-1}$.

We will argue by induction. The case $l=1$ is Lemma \ref{lemaa} with $r=d-1$. Now, assume the result for some $l<d$, and for $1\leq j \leq \binom{d}{l+1}$, define $\hat{T}_j=T_{r(l)j}\cdots T_{r(l)(j-1)+1}$. Then $(\hat{T}_1,\dots,\hat{T}_{\binom{d}{l+1}})\in \Nil^{\binom{d}{l+1}}(V)$, and by our inductive hypothesis, we obtain $\rank(\hat{T}_j)\leq d-l-1$. So, we are in the assumption of \ref{lemaa} with $r=d-l-1$ and we conclude that $\rank(T_{r(l+1)}\cdots T_{1})=\rank(\hat{T}_{\binom{d}{l+1}}\cdots \hat{T}_{1})<d-l-1$. This proves the claim and concludes the proof of the theorem.  
\end{proof}

\subsection{Some computations in low dimension}\label{computos}
Let $N(d)$ be the least value of $N$ for which Theorem \ref{nilpotentes} (and therefore also Theorem \ref{desigualdad}) holds true. From the proof of Theorem \ref{nilpotentes}, we can obtain the bound $N(d)\leq \binom{d}{1}\binom{d}{2}\cdots \binom{d}{d-1}$ for all $d$. Also, since for all $d$ we can construct a matrix $A\in \mathcal{N}_{d}(k)$ of rank $d-1$, the tuple $(A,\dots,A)\in \mathcal{N}^{d-1}(k^{d})$ satisfies $A^{d-1}\neq 0$ and hence we have the lower bound $N(d)\geq d$. In particular, we conclude that $N(2)=2$, and for higher dimensions we get the bounds $3\leq N(3)\leq 9$ and $4\leq N(4)\leq 96$. We end this section by finding a better bound for $N(3)$.

\begin{prop}\label{casotres}
For any field $k$, we have $N(3)\leq 5$. In addition, if $\Char{k}\neq 2$, then $N(3)=5$.
\end{prop}
To prove this, we need a lemma:
\begin{lema}\label{lema3}
Let $(C,B,A)\in \mathcal{N}^{3}(k^{3})$. If $\rank{B}=1$, then $AB=\lambda B$ or $BC=\lambda B$ for some $\lambda \in k$.
\end{lema}
\begin{proof}
Assume that $B=\begin{pmatrix}
0 & 0 & 0\\
0 & 0 & 1\\
0 & 0 & 0\\
\end{pmatrix}$, $A=\begin{pmatrix}
a & b & c\\
d & e & f\\
g & h & i\\
\end{pmatrix}$ and $C=\begin{pmatrix}
p & q & r\\
s & t & u\\
v & w & x\\
\end{pmatrix}$. Then 
$AB=\begin{pmatrix}
0 & 0 & b\\
0 & 0 & e\\
0 & 0 & h\\
\end{pmatrix}$ and $BC=\begin{pmatrix}
0 & 0 & 0\\
v & w & x\\
0 & 0 & 0\\
\end{pmatrix}$. The nilpotence of $AB$ and $BC$ implies $h=\Tr{AB}=w=\Tr{BC}=0$. Then $ABC=\begin{pmatrix}
bv & 0 & bx\\
ev & 0 & ex\\
0 & 0 & 0\\
\end{pmatrix}$, and by the nilpotence of $ABC$, $bv=\Tr{ABC}=0$.
The case $b=0$ is $AB=eB$ and the case $v=0$ is $BC=xB$.
\end{proof}

\begin{coro}\label{cororango3}
If $(C,B,A) \in \mathcal{N}^{3}(k^{3})$ and $\rank(B)\leq 1$, then $ABC=0$. 
\end{coro}
\begin{proof}
Assume that $\rank(B)=1$, and $AB, BC\neq 0$. By Lemma \ref{lema3} and after rescaling $A$ or $C$, we may suppose that $BC=B$ or $AB=B$. In the first case we will have $(C^2,B,A) \in \mathcal{N}^{3}(k^{3})$, and by Corollary \ref{cororango}, $ABC=A(BC)C=ABC^{2}=0$. For the case $AB=B$, applying a similar argument to the tuple $(A^{t},B^{t},C^{t})$ of the transposes of $A,B,C$, we will obtain $(ABC)^{t}=C^{t}B^{t}(A^{t})^{2}=0$, and hence $ABC=0$.
\end{proof}

\begin{proof}[Proof of Proposition \ref{casotres}]
Let $(E,D,C,B,A)\in \Nil^{5}(k^{3})$. Then $(E,BCD,A)$ belongs to $\Nil^{3}(k^{3})$, and by Lemma \ref{lemaa} with $d=3, r=2$, $\rank(BCD)\leq 1$. Then, by Corollary \ref{cororango3}, $ABCDE=0$ and $N(3)\leq 5$. 
Moreover, when $\Char{k}\neq 2$, it is a straightforward computation to show that $(D,C,B,A)\in \Nil^{4}(k^{3})$, with
\small\begin{equation}
\hspace{-1mm}A=\begin{pmatrix}
-2 & -6 & 1\\
3 & 9 & 16\\
-1 & -3 & -7\\
\end{pmatrix},B=\begin{pmatrix}
0 & 1 & 0\\
0 & 0 & 1\\
0 & 0 & 0\\
\end{pmatrix},C=\begin{pmatrix}
1 & 1 & 0\\
1 & 1 & 4\\
-1 & -1 & -2\\
\end{pmatrix},D=\begin{pmatrix}
-1 & 3 & 16\\
1 & -3 & -16\\
1 & 2 & 4\\
\end{pmatrix} \notag
\end{equation}
and $ABCD=\begin{pmatrix}
4 & 8 & 16\\
-6 & -12 & -24\\
2 & 4 & 8\\
\end{pmatrix}\neq 0$.
\end{proof}

\begin{obs} This last proposition shows that, in general we cannot expect $N(d)=d$. For that reason, the hypothesis of Theorem \ref{nilpotentes} does not imply any kind of simultaneous triangularization or nilpotency of the semigroup generated by sequences in $\mathcal{N}^{n}(k^{d})$. In fact, it is not hard to prove that the matrices in the last example we gave in $\mathcal{N}^{4}(k^{3})$ are not simultaneously triangularizable, since $A$ and $B$ do not have a common invariant subspace of dimension $1$.
\end{obs}
\section{A polynomial identity}\label{secciontres}

For the proof of Theorem \ref{desigualdad} we need some notation. Let $k$ be a field with algebraic closure $\overline{k}$. For $d,N\in \N$, consider $Nd^2$ variables $x_{i,j}$ with $1\leq i\leq N$, $1 \leq j \leq d^{2}$ and let $R_{d,N}$ be the polynomial ring $k[x_{i,j}]$. If $A_1,\dots,A_N \in M_{d}(\overline{k})$ and $f\in R_{d,N}$, by $f(A_1,\dots,A_N)$ we mean the element $f((a_{i,j})_{i,j})$ where $(a_{i,j})_{j}$ are the coefficients of $A_i$ in some fixed order.

Recall that a polynomial $f\in k[y_1,\dots,y_m]$ is \textit{homogeneous} of degree $\lambda\geq 0$ if it is of the form $\sum_{i_1+\dots+i_m=\lambda}{c_{i_1\dots i_m}y_{1}^{i_1}\dots y_{m}^{i_m}}$ for some $c_{i_1\dots i_m}\in k$, $i_1,\dots,i_m\geq 0$. We say that monomial $f\in R_{d,N}$ is \textit{multihomogeneous} of \textit{multidegree} $\mdeg{f}=(\lambda_1,\dots,\lambda_N) \in {\N}^{N}$ if it is of the form $f((x_{i,j})_{i,j})=c\prod_{i,j}x^{u_{i,j}}_{i,j}$, where $c\in k$, $u_{i,j}\geq 0$ and $\sum_{j}{u_{i,j}}=\lambda_{i}$ for all $1\leq i\leq N$, and that a polynomial $p\in R_{d,N}$ is multihomogeneous of multidegree $\mdeg{p}$ if it is a finite sum of multihomogeneous monomials of multidegree $\mdeg{p}$. This is equivalent to say that, for each $1\leq i\leq N$, $p$ is homogeneous of degree $\lambda_i$ in the variables $x_{i,1} \dots x_{i,d^2}$. Note there is a direct sum decomposition \begin{equation}\label{directsum}
R_{d,N}=\bigoplus_{\lambda \in {\N}^{N}}{R_{d,N,\lambda}},
\end{equation} where $R_{d,N,\lambda}$ denotes the vector space of multihomogeneous polynomials of multidegree $\lambda$.

For $1\leq j\leq d^2$ denote by $f_{j}$ the polynomial in $R_{d,N}$ representing the map that sends the $N$-tuple $(A_1,\dots,A_N)\in \overline{k}^{Nd^2}$ to the $j$-th entry of $A_N\cdots A_1$. Also, for $1\leq \ell \leq d$ and $1\leq \alpha \leq \beta \leq N$, let $T^{\ell}_{\alpha,\beta} \in R_{d,N}$ be the polynomial that represents the map $(A_1,\dots,A_N)\mapsto \Tr{\Wedge^{\ell}(A_{\beta}\cdots A_{\alpha})}$. 

It is not hard to see that $f_{j}$ are multihomogeneous of multidegree $(1,1,\dots,1,1)$ and that $T_{\alpha,\beta}^{\ell}$ are multihomogeneous of multidegree $(0,\dots,0,\ell,\dots,\ell,0,\dots,0)$, with the $\ell$'s in positions $\alpha,\alpha+1,\dots,\beta$.

Our purpose is to prove the following:
\begin{teo}\label{traduccion}
If $N=N(d)$ is given by Theorem \ref{nilpotentes}, there is some $r\leq (Nd+1)^{Nd^2+2}$ such that for all $1\leq j\leq d^2$ there exist multihomogeneous polynomials $p^{\alpha,\beta}_{j,\ell} \in R_{d,N}$ of multidegree $r\mdeg{f_{j}}-\mdeg{T^{\ell}_{\alpha,\beta}}\in {\N}^{N}$ such that 
\begin{equation}\label{idpolinomial}
    (f_{j})^{r}=\sum_{\alpha,\beta,\ell}{p^{\alpha,\beta}_{j,\ell}T^{\ell}_{\alpha,\beta}}.
\end{equation}
\end{teo}
The natural tool to prove this result is Hilbert's Nullstellensatz. If $I\subset k[y_1,\dots,y_m]$ is an ideal, let $Z(I)$ be its zero locus in ${\overline{k}}^m$. Also, for $Z\subset \overline{k}^m$, let $I(Z)\in k[y_1,\dots,y_m]$ be the ideal of polynomials $f$ that vanish on $Z$. The effective version of  Nullstellensatz that we will use comes from applying Rabinowitsch's proof of Nullstellensatz by assuming weak Nullstellensatz (see e.g. \cite[Sec.~1.7]{fulton}) and an effective version of weak Nullstellensatz \cite{som}):

\begin{teo}[Effective Nullstellensatz]\label{effnul} If $I=\subset k[y_1,\dots,y_m]$ is an ideal, then $I(Z(I))$ is equal to $\sqrt{I}$, the radical ideal of $I$. Moreover, if $g\in \sqrt{I}$ and $I$ is generated by polynomial $f_1,\dots,f_s$ satisfying $\max(\deg{f_1},\dots,\deg{f_s},\deg{g})=k$, then there is some $r\leq (k+1)^{m+2}$ and $p_1,\dots,p_s\in k[y_1,\dots,y_m]$ such that
	\begin{equation*}
	g^r=p_1f_1+p_2f_2+\dots+p_sf_s.
	\end{equation*} 
\end{teo}
\begin{proof}[Proof of Theorem \ref{traduccion}]
Let $I\subset R_{d,N}$ be the ideal generated by the polynomials $T_{\alpha,\beta}^{\ell}$ and let $W=Z(I)$.
%Given $\alpha, \beta,\ell$ and $\gamma=(j_1,\dots,j_{\alpha-1},j_{\beta+1},\dots,j_{N})\in \corchete{1,\dots,d^2}^{N-\beta+\alpha-1}$, define $u_{\gamma}\in R_{d,N}$ as $u_{\gamma}(A_1,\dots,A_N)=(a_{1,j_1}\cdots a_{(\alpha-1),j_{\alpha-1}}a_{(\beta+1),j_{\beta+1}}\cdots a_{N,j_N})^{\ell}$, with the convention $a_{1,j_1}\cdots a_{0,j_{0}}=a_{(N+1),j_{N+1}}\cdots a_{N,j_N}=1$. Also, let $S_{\alpha,\beta,\gamma}^{\ell}\in k[z_{i_1\dots i_N}]$ be the homogeneous polynomial of degree $\ell$ such that 
%\begin{equation}\label{projalg}
%S_{\alpha,\beta,\gamma}^{\ell}(\varphi(A_1,\dots,A_N))=T_{\alpha,\beta}^{\ell}(A_1,\dots,A_N)(u_{\gamma}(A_1,\dots,A_{N}))^{\ell} 
%\end{equation}
%for all $A_1,\dots,A_N \in \overline{k}^{d^2}$.
%In a similar way, define $g_{j}\in k[z_{i_1\dots i_N}]$ as the homogeneous polynomial of degree 1 such that $g_{j}\circ \varphi=f_{j}$.
%It is clear that, for $P\in (\mathbb{P}^{d^2-1})^{N}$, $T_{\alpha,\beta}^{\ell}(P)=0$ if and only if $S_{\alpha,\beta,\gamma}^{\ell}(\hat{\varphi}(P))=0$ for all $\gamma$. We deduce from this that $\hat{\varphi}(W)=Z(I')\cap \Ima{\hat{\varphi}}=Z(I'+J)$, where $I'\subset k[z_{i_1\cdots i_N}]$ is the homogeneous ideal generated by the polynomials $S_{\alpha,\beta,\gamma}^{\ell}$. 
Note that for a matrix $A$ of order $d\times d$, the non leading coefficients of its characteristic polynomial are precisely $(-1)^{\ell}\Tr{\Wedge^{\ell}(A)}$, with $1\leq \ell \leq d$. By this observation, the set $W$ is precisely the set of $N$-tuples $(A_1,\dots,A_N)\in \mathcal{N}_{d}^{N}(\overline{k})$. Hence, by our choice of $N$, Theorem \ref{nilpotentes} guarantees us that $f_{j}(P)=0$ for all $P\in W$. Then Nullstellensatz applies and $f_{j}\in I(Z(W))=\sqrt{I}$.

Since $\max(\deg{f_j},\deg{T_{\alpha,\beta}^{\ell}})=Nd$, by Theorem \ref{effnul} there is some $r\leq (Nd+1)^{Nd^2+2}$ and polynomials $p_{j,\ell}^{\alpha,\beta,\gamma}\in R_{d,N}$ satisfying \eqref{idpolinomial} for all $j$. Finally, by the direct sum decomposition \eqref{directsum} and by comparing multidegrees, we may assume that $p^{\alpha,\beta,\gamma}_{j,\ell}$ are multihomogeneous of multidegree $r\mdeg{f_{j}}-\mdeg{T^{\ell}_{\alpha,\beta,\gamma}}$.
\end{proof}

\section{Proof of Theorem \ref{desigualdad}}\label{pruebadesigualdad}

Theorem \ref{traduccion} is the fundamental relation that we will need to prove inequality \eqref{eqdesigualdad}.

For the next we will assume that $k$ is a \textit{local field}. That is, a field together with an absolute value $|.|:k\rightarrow \R^{+}$ that induces a non-discrete locally compact topology on $k$ via the induced metric. Examples of these include $\R,\C$ with the standard absolute values and fields of $p$-adic numbers $\mathbb{Q}_{p}$ for a prime $p$. For more information about local fields, see \cite{lorentz}.

We will work on the finite dimensional vector space $k^d$, where $k$ is a local field with absolute value $|.|$. In this situation, we consider the norm on $M_{d}(k)$ given by $\|A\|_{0}=\max_{1\leq j\leq d^2}{|a_{j}|}$, where $a_j$ are the entries of $A$. Since the absolute value on $k$ extends in a unique way to an absolute value on $\overline{k}$ (see Lang's Algebra \cite[XII.2, Prop.~2.5]{langalg}), the spectral radius of a matrix $A\in M_{d}(k)$ is then defined in the usual way. 
The \emph{height} $h(f)$ of a polynomial $f\in k[y_1,\dots,y_m]$ is defined as the logarithm of the maximum modulus of its coefficients.

We begin with a lemma.

\begin{lema}\label{cotahomo}If $f \in R_{d,N}$ is a multihomogeneous polynomial of multidegree $(\lambda_1,\dots,\lambda_{N})$ and height $h(f)\leq H$, then 
\begin{equation}
|f(A_1,\dots,A_N)|\leq e^H\prod_{i=1}^{N}{\binom{d^2-1+\lambda_i}{\lambda_i}}\|A_1\|_{0}^{\lambda_1}\cdots\|A_{N}\|_{0}^{\lambda_N} \notag 
\end{equation}
for all $A_1,\dots,A_N\in M_{d}(k)$.
\end{lema}

\begin{proof}
If $f$ is a multihomogeneous monomial with $h(f)\leq H$, then
 $f(X_1,\dots,X_N)=c\prod_{i=1}^{N}{\prod_{j=1}^{\lambda_{i}}{X_{i ,\ell_{i,j}}}}$, for some $1\leq\ell_{i,j} \leq d^2$ and $c\in k$ with $|c|\leq e^H$. So, given $A_1,\dots,A_N\in M_{d}(k)$,
\begin{equation}
|f(A_1,\dots,A_N)|=|c|\prod_{i=1}^{N}{\prod_{j=1}^{\lambda_{i}}{|A_{i ,\ell_{i,j}}}|}\leq e^H\prod_{i=1}^{N}{\|A_i\|_{0}^{\lambda_{i}}}.   \notag
\end{equation}
The lemma then follows by noting that a multihomogeneous polynomial of multidegree $(\lambda_1,\dots,\lambda_N)$ is sum of at most $\prod_{i=1}^{N}{\binom{d^2-1+\lambda_i}{\lambda_i}}$ multihomogeneous monomials of the same multidegree.
\end{proof}

\begin{proof}[Proof of Theorem \ref{desigualdad}]
Let $N=N(d)$ and $r>1$ be given by Theorems \ref{nilpotentes} and \ref{traduccion} respectively, and consider first $n=N$ and the norm $\|.\|_{0}$. Let $A_1,\dots, A_N \in M_{d}(k)$. First, note that for $1\leq \alpha\leq \beta\leq N$ and $1\leq \ell\leq d$, $T_{\alpha,\beta}^{\ell}(A_1,\dots,A_N)$ is the $\ell$-th symmetric polynomial evaluated at the eigenvalues of $A_{\beta}\cdots A_{\alpha}$. Hence we have $|T_{\alpha,\beta}^{\ell}(A_1,\dots,A_N)|\leq \binom{d}{\ell}\rho(A_{\beta}\cdots A_{\alpha})^{\ell}$. Also, as the polynomials $p^{\alpha,\beta}_{j,\ell}$ in the statement of Theorem \ref{traduccion} have multidegree $(r,\dots,r,r-\ell,\dots,r-\ell,r,\dots,r)$,  by Lemma \ref{cotahomo} we have \begin{equation*}
	\begin{split}
|p^{\alpha,\beta}_{j,\ell}(A_1,\dots,A_N)| & \leq e^H{\binom{d^2-1+r}{r}}^{N-\beta+\alpha-1}{\binom{d^2-1+r-\ell}{r-\ell}}^{\beta-\alpha+1} \\
& \quad \times \left(\prod_{s=1}^{N}{\|A_s\|_{0}}\right)^{r}\left(\prod_{t=\alpha}^{\beta}{\|A_{t}\|_{0}}\right)^{-\ell}
	\end{split}
\end{equation*} for all $j,\alpha,\beta,\ell$.   
Thus, from \eqref{idpolinomial} we obtain the following:
\begin{equation}
\begin{split}
\|A_N\cdots A_1\|_{0}^{r} & = \max_{j}|f_{j}(A_1,\dots,A_N)|^{r} \\
& \leq \max_{j} \sum_{\alpha,\beta,\ell}|p_{j,\ell}^{\alpha,\beta}(A_1,\dots,A_N)||T_{\alpha,\beta}^{\ell}(A_1,\dots,A_N)| \\
 & \leq C_{1}\left(\prod_{i=1}^{N}{\|A_i\|_{0}}\right)^{r}\max_{\alpha,\beta,\ell}{\left(\frac{\rho(A_{\beta}\cdots A_{\alpha})}{\prod_{t=\alpha}^{\beta}{\|A_{t}\|_{0}}}\right)^{\ell}},
\end{split}  \notag
\end{equation}
for $C_1=e^H\sum_{\alpha,\beta,\ell}{{\binom{d^2-1+r}{r}}^{\small{N-\beta+\alpha-1}}{\binom{d^2-1+r-\ell}{r-\ell}}^{\small{\beta-\alpha+1}}\binom{d}{\ell}}$. 

Now, let $\Lambda=\max_{\alpha,\beta}{\left(\frac{\rho(A_{\beta}\cdots A_{\alpha})}{\prod_{t=\alpha}^{\beta}{\|A_{t}\|_{0}}}\right)}$. An easy computation shows that $\|AB\|_{0}\leq d\|A\|_{0}\|B\|_{0}$ for all $A,B\in M_{d}(k)$. Moreover, by the Gelfand's formula $\rho(A)=\lim_{n\to \infty}{{\|A^n\|_0}^{1/n}}$ we obtain $\rho(A)\leq d\|A\|_{0}$ for all $A\in M_{d}(k)$. These facts together imply that $\Lambda\leq d^{N+1}$, and hence $\Lambda^{d}\leq d^{(N+1)(d-1)}\Lambda$. Also, depending on whether $\Lambda$ is greater than $1$ or not, we have $\Lambda^{\ell}\leq \max(\Lambda,\Lambda^{d})\leq d^{(N+1)(d-1)}\Lambda$ for all $1\leq \ell \leq d$. Thus we conclude 
\begin{equation}
\|A_N\cdots A_1\|_{0}^{r}   \leq d^{(N+1)(d-1)}C_1\left(\prod_{i=1}^{N}{\|A_i\|_{0}}\right)^{r}\cdot\Lambda. \notag
\end{equation}
Applying $r$-th root to the last inequality, we obtain \eqref{eqdesigualdad} with $C=C(\|.\|_0):={C_1}^{1/r}d^{(N+1)(d-1)/r}$.

Now, consider an arbitrary submultiplicative norm $\|.\|$ on $M_d(k)$. Since in a finite dimensional vector space over a local field all norms are equivalent \cite[XII.2, Prop.~2.2]{langalg}, there is some $D\geq 1$ such that $D^{-1}\|A\| \leq \|A\|_0\leq  D\|A\|$ for all $A\in M_d(k)$. This implies
\begin{equation}
\begin{split}
\|A_N\cdots A_1\| & \leq D\|A_N\cdots A_1\|_{0} \\
& \leq DC \left(\prod_{1\leq i\leq N}{\|A_i\|_0}\right)\max_{1\leq \alpha\leq \beta\leq N}{ \left( \frac{\rho(A_{\beta}\cdots A_{\alpha}) }{ \prod_{\alpha\leq i\leq \beta}{\|A_i\|_0}} \right)^{1/r}} \\
& \leq D^{N+1}C\left(\prod_{1\leq i\leq N}{\|A_i\|}\right)\max_{1\leq \alpha\leq \beta\leq N}{ \left( \frac{\rho(A_{\beta}\cdots A_{\alpha}) }{ \prod_{\alpha\leq i\leq \beta}{\|A_i\|}} \right)^{1/r}} \\ 
\end{split}  \notag
\end{equation}
and proves the statement for $n=N$. For a general $n\geq N$ the result follows by applying  \eqref{eqdesigualdad} to the sequence $A_1,\dots,A_{N-1},A_NA_{N+1}\cdots A_n$, and then using the submultiplicativity of $\|.\|$.
\end{proof}
\subsection{The case of the complex numbers}\label{casoc} When the base field is $k=\C$ we can say a little more. Recall that for a norm $\|.\|$ on $\C^{d}$, the \textit{operator norm} on $M_{d}(\C)$ (also denoted by $\|.\|$) is defined by $\|A\|=\sup_{v\in \C^{d}\backslash \{ 0\}}{\frac{\|Av\|}{\|v\|}}$.
\begin{prop}\label{independ} For $d\in \N$ and $N(d)$ given by Theorem \ref{desigualdad}, there is a constant $C=C(d)>1$ such that if $n\geq N$, the inequality \eqref{eqdesigualdad} holds for all operator norms $\|.\|$ on $M_{d}(\C)$ and $A_1,\dots,A_n\in M_{d}(\C)$.
\end{prop}
We will need the following lemma which is a consequence of John's ellipsoid theorem \cite[Thm.~10.12.2]{schr} (see also \cite[Lemma 3.2]{bociineq}):
\begin{lema}\label{lemadeljairo}
For all $d\in \N$ and for every two operator norms $\|.\|$ and $\|.\|_{1}$ on $M_{d}(\C)$ there exists some $S\in\mathrm{GL}_{d}(\C)$ such that for every $A \in M_{d}(\C)$:
\begin{equation}\label{kk}
d^{-1}\|A\|\leq\|SAS^{-1}\|_{1}\leq d\|A\|. \notag
\end{equation}
\end{lema}
\begin{proof}[Proof of Proposition \ref{independ}]
Since operator norms are submultiplicative, we may assume that $n=N$. Fix the $\ell^1$-operator norm $\|A\|_{1}=\max_{j}{\sum_{i=1}^{d}|a_{ij}|}$ on $M_{d}(\C)$, where $a_{ij}$ are the entries of $A$. Since $\|A\|_1 \leq \|A\|_0 \leq d\|A\|_1$ for all $A$, by the proof of Theorem \ref{desigualdad} we have $C(\|.\|_1)=d^{N+1}C(\|.\|_0)$ with $C(\|.\|_0)$ as in the end of the proof of Theorem \ref{desigualdad}. If $\|.\|$ is an arbitrary operator norm on $M_{d}(\C)$, let $S\in \mathrm{GL}_{d}(\C)$ be relating $\|.\|$ and $\|.\|_{1}$ as in Lemma \ref{lemadeljairo}. 

Given $A_{1},\dots,A_{N}\in M_{d}(\C)$ let $B_{i}=SA_{i}S^{-1}$ for all $i$. We have
\begin{equation}
\begin{split}
\|A_N\cdots A_1\| & \leq d\|B_N\cdots B_1\|_{1} \\
& \leq d^{N+2}C(\|.\|_0) \left(\prod_{1\leq i\leq N}{\|B_i\|_1}\right)\max_{1\leq \alpha\leq \beta\leq N}{ \left( \frac{\rho(B_{\beta}\cdots B_{\alpha}) }{ \prod_{\alpha\leq i\leq \beta}{\|B_i\|_1}} \right)^{1/r}} \\
& \leq d^{2N+2}C(\|.\|_0)\left(\prod_{1\leq i\leq N}{\|A_i\|}\right)\max_{1\leq \alpha\leq \beta\leq N}{ \left( \frac{{d}^{\beta-\alpha+1}\rho(A_{\beta}\cdots A_{\alpha}) }{ \prod_{\alpha\leq i\leq \beta}{\|A_i\|}} \right)^{1/r}} \\ 
 & \leq d^{3N+2}C(\|.\|_0)\left(\prod_{1\leq i\leq N}{\|A_i\|}\right)\max_{1\leq \alpha\leq \beta\leq N}{ \left( \frac{(\rho(A_{\beta}\cdots A_{\alpha}) }{ \prod_{\alpha\leq i\leq \beta}{\|A_i\|}} \right)^{1/r}}.
\end{split}  \notag
\end{equation}
It is clear that $C(d):=d^{3N+2}C(\|.\|_0)$ does not depend on $\|.\|$.
\end{proof}
Proposition \ref{independ} allows us to conclude the following inequality:
\begin{teo}\label{detodos} Given $d\in \N$, let $C(d)$ be as in Proposition \ref{independ}. Then the following inequality is valid for all bounded sets $\mathcal{M}\subset M_{d}(\mathbb{C})$:
\begin{equation}
\mathfrak{R}(\mathcal{M})\leq C(d)\max_{1\leq j\leq N(d)}{\left(\sup\corchete{\rho(A_{1}\cdots A_{j}) : A_{i}\in \mathcal{M}}\right)^{1/j}}. \notag
\end{equation}
\end{teo}
This inequality was first proved by Bochi in \cite{bociineq} (without giving and effective bound on $N(d)$), and it has Theorem \ref{bergerwangantiguo} as an immediate consequence. In \cite[Lemma 2.1]{bro}, Breuillard gave another proof of this inequality valid for arbitrary local fields: in the Arquimedean case the same conclusion of Theorem \ref{detodos} holds with $N(d)=d^2$, while for a non Arquimedean field $k$ and for every bounded set $\mathcal{M}\subset M_d(k)$ we have the identity
\begin{equation}
\inf_{S\in \mathrm{GL}_d(k)}{\|S\mathcal{M}S^{-1}\|_0}=\max_{1\leq j\leq d^2}{\left(\sup\corchete{\rho(A_{1}\cdots A_{j}) : A_{i}\in \mathcal{M}}\right)^{1/j}}. \notag
\end{equation}
Breuillard used this result to study semigroups of invertible matrices.
\begin{proof}[Proof of Theorem \ref{detodos}]
For $1\leq j\leq N(d)$, define $\rho_{j}=\sup\corchete{\rho(A_{1}\cdots A_{j}) : A_{i}\in \mathcal{M}}$. For an arbitrary operator norm $\|.\|$ on $M_{d}(\C)$, take supremum for $A_{i}\in \mathcal{M}$ in both sides of \eqref{eqdesigualdad}. We obtain
\begin{alignat}{2}
\mathfrak{R}(\mathcal{M})^{N} & \leq\sup_{A_{i}\in \mathcal{M}}{\|A_{N}\cdots A_{1}\|}  \notag\\
& \leq C(d) \max_{1\leq j\leq N}\left(\sup_{A\in \mathcal{M}}{\|A\|}^{N-j/r}\cdot (\rho_{j})^{1/r}\right). &\quad \label{keke}
\end{alignat}
Now, recall that $\mathfrak{R}(\mathcal{M})=\inf_{\|.\|}{\sup_{A\in \mathcal{M}}{\|A\|}}$, where the infimum is taken over all operator norms on $M_{d}(\C)$ (for a proof, see \cite{rota}), and let $\|.\|_{n}$ be a sequence of operator norms on $M_{d}(\C)$ such that $\sup_{A\in \mathcal{M}}{\|A\|_{n}}\to \mathfrak{R}(\mathcal{M})$. Taking a subsequence, we may assume that for all $\|.\|_{n}$, the maximum in the right hand side of \eqref{keke} is achieved by the same index $j\in \{1,\dots, N \}$. Then, taking limit as $n$ tends to infinity in \eqref{keke} we will have
\begin{equation}\label{kiki}
\mathfrak{R}(\mathcal{M})^{N}\leq C(d)(\rho_{j})^{1/r}\cdot\mathfrak{R}(\mathcal{M})^{N-j/r }
\end{equation}
(here is where we use Proposition \ref{independ} since $C(d)$ does not depend on $n$). If $\mathfrak{R}(\mathcal{M})=0$ the conclusion is obvious. Otherwise, dividing by $\mathfrak{R}(\mathcal{M})^{N-j/r}$ and taking $j/r$-th root in \eqref{kiki} we obtain the desired inequality. 
\end{proof}

\begin{obs}\label{effective} In the proof of Theorem \ref{desigualdad} the constant $C(\|.\|_0)$ depends on $H$, the maximum on the heights of the polynomials $p_{j,\ell}^{\alpha,\beta}$. In the case of the complex numbers we can give an effective upper bound on $H$ by means of the effective arithmetic Nullstellensatz \cite[Thm.~1]{klpasom}. Applied to our case, we obtain $r\leq 4(Nd^2+1)(dN)^{Nd^2+1}$ and
	\small{\begin{equation*}
	H\leq \log{2}(Nd^2+(N(d^2+1))r)+r(Nd^2+2)(\log(d\dbinom{N+1}{2}+1)+(Nd^2+8)\log(Nd^2+2)dN).
	\end{equation*}}
\end{obs} \normalsize
\vspace{-2ex}Hence \small{\begin{equation*}
	\begin{split}
C(d)& =d^{3N+2}C(\|.\|_0) \\
& =e^{H/r}d^{3N+2+(N+1)(d-1)/r}\left( \sum_{\alpha,\beta,\ell}{{\binom{d^2-1+r}{r}}^{\small{N-\beta+\alpha-1}}{\binom{d^2-1+r-\ell}{r-\ell}}^{\small{\beta-\alpha+1}}\binom{d}{\ell}}\right)^{1/r}.
\end{split}
\end{equation*}} \normalsize
\section{Ergodic-theoretical consequences}\label{seccioncinco}
For the proof of Theorem \ref{bergerwangnuevo}, we will need the following result which may be seen as a quantitative version of Poincar\'e's Recurrence theorem for measure preserving transformations. It is a consequence of Birkhoff Ergodic Theorem, and the fact that for a measurable set $U$ of positive measure, for almost all points $x$ in $U$, the frequency of points of the sequence $x,Tx,T^{2}x,\dots$ that belong to $U$ is positive (compare with the subbaditive ergodic theorem of Karlsson-Gou\"ezel \cite[Thm.~1.1]{goukar}). For a detailed proof, see \cite[Lemma 3.12]{bochilya}.
\begin{lema}\label{superpoincare}
Let $T:X\rightarrow X$ be a measure preserving map over the probability space $(X,\mathcal{F},\mu)$, and let $U\in \mathcal{F}$ have positive measure. Given $\gamma>0$, there exists a measurable map $N_0:U\rightarrow \N$ such that, for $\mu-$a.e. $x\in U$ and for every $n\geq N_{0}(x)$ and $t\in [0,1]$ there is some $\ell \in \corchete{1,\dots,n}$ with $T^{\ell}(x)\in U$ and $|(\ell/n)-t|<\gamma$.
\end{lema}
\begin{proof}[Proof of Theorem \ref{bergerwangnuevo}]
Fix an operator norm $\|.\|$ on $M_{d}(k)$, and let $Y=\{x\in X : \lambda(x)\in \R\}$. This is a measurable $T$-invariant set, and since $\rho(A)\leq \|A\|$ for all $A\in M_{d}(k)$, we have that both sides of\eqref{bww} equal $-\infty$ for $\mu$-almost all $x\in X\backslash Y$. So we only have to check the result $\mu-$a.e. in $Y$.

Assume the contrary. That is, assume the existence of some $\epsilon>0, K\in \N$ and a measurable set $U\subset Y$ of positive measure such that, for all $x\in U$, if $n\geq K$, then $\log{\rho(\ccl{A}{n}{x})}/n+\epsilon \leq \lambda(x)$.
By Egorov's theorem, and restricting to a smaller subset if necessary, we may assume that on $U$,  $\log{\|\ccl{A}{n}{x}\|}/n$ converges uniformly to $\lambda(x)$.

Let $N,r$ and $C$ be as in the statement of Theorem \ref{desigualdad} and let $\epsilon'=\epsilon/(2+6Nr)$. By the uniform convergence assumption, there is some $M\geq 1$ such that, $ n\geq M$ implies 
\begin{equation}\label{necesaria1}
|\log{\|\ccl{A}{n}{x}\|}-n\lambda(x)|<n\epsilon' \text{ for all }x\in U.
\end{equation}
Take $x\in U$ and $N_0(x)\in \N$ such that Lemma \ref{superpoincare} holds with $\gamma=1/3N$, and let $n\geq \max(3NM,3NK,3Nr\log{C}/\epsilon',N_0(x))$. Let $m_0=0$, and given $1\leq i\leq N$ let $1\leq m_i \leq n$ be such that 
\begin{equation}\label{fracrecu}
\left|\frac{m_i}{n}-\frac{i}{N}\right|<\frac{1}{3N} 
\end{equation}
and $T^{m_i}x\in U$. We have that $m_{i}-m_{i-1}>(in/N-n/3N)-((i-1)/N+n/3N)=n/3N\geq \max(M,K,r\log{C}/\epsilon')$ for all $1\leq i\leq N$.

Now apply Theorem \ref{desigualdad} to $A_{i}=\ccl{A}{m_i-m_{i-1}}{T^{m_{i-1}}x}$. By the cocycle relation, we obtain $A_{N}\cdots A_{1}=\ccl{A}{m_N}{x}$, and hence
\small\begin{equation}\label{necesaria2}
\begin{split}
\hspace{-3mm}\log{\|\ccl{A}{m_N}{x}\|}  & \leq  \log{C}+\sum_{i=1}^{N}{\log{\|\ccl{A}{m_i-m_{i-1}}{T^{m_{i-1}}x}\|}}  \\
 & \hspace{3mm}+ \frac{1}{r}\left(\log{\rho(\ccl{A}{m_{\beta}-m_{\alpha-1}}{T^{m_{\alpha-1}}x})}-\sum_{i=\alpha}^{\beta}{\log{\|\ccl{A}{m_i-m_{i-1}}{T^{m_{i-1}}x}\|}}\right) 
 \end{split}
\end{equation}
\normalsize for some $1\leq \alpha\leq \beta\leq N$. But, by definition, $T^{m_i}x\in U$ for all $i$, and as $m_i-m_{i-1}\geq M$, \eqref{necesaria1} applies. Combining it with \eqref{necesaria2} we have
\begin{equation}
\begin{split}
\log{\rho(\ccl{A}{m_{\beta}-m_{\alpha-1}}{T^{m_{\alpha-1}}x})} & \geq
\sum_{i=\alpha}^{\beta}{\log{\|\ccl{A}{m_i-m_{i-1}}{T^{m_{i-1}}x}\|}} \\
& +r\left( \log{\|\ccl{A}{m_N}{x}\|}-\sum_{i=1}^{N}{\log{\|\ccl{A}{m_i-m_{i-1}}{T^{m_{i-1}}x}\|}} -\log{C}  \right)\\
&  >   (m_{\beta}-m_{\alpha-1})(\lambda(x)-\epsilon')-r(\log{C}+2\epsilon'm_N)\\
& = (m_{\beta}-m_{\alpha-1})\lambda(x)-\left( \epsilon'((m_{\beta}-m_{\alpha-1})+2rm_N))+r\log{C}    \right). \notag
\end{split}
\end{equation}
On the other hand, by \eqref{fracrecu} we have 
\begin{equation}
\frac{m_N}{m_{\beta}-m_{\alpha-1}}<\frac{n}{\frac{n(\beta-\alpha+1)}{N}-\frac{2n}{3N}}\leq 3N. \notag
\end{equation}
But, since $T^{m_{\alpha-1}}x\in U$, and $(m_{\beta}-m_{\alpha-1})\geq \max(K,r\log{C}/\epsilon')$ we conclude
\begin{equation}
\begin{split}
\frac{\epsilon'((m_{\beta}-m_{\alpha-1})+2r(m_N))+r\log{C} }{m_{\beta}-m_{\alpha-1}} & = \epsilon'+\frac{\epsilon'2r(m_N)}{(m_{\beta}-m_{\alpha-1})}+\frac{r\log{C}}{(m_{\beta}-m_{\alpha-1})} \\
 & \leq \epsilon'+\epsilon'6Nr+\frac{r\log{C}}{(m_{\alpha}-m_{\alpha-1})}\\
  & \leq (2+6Nr)\epsilon'=\epsilon. \notag
\end{split}
\end{equation}
This is the desired contradiction and the proof is complete.
\end{proof}

\section{Geometric remarks}\label{seccionseis}

We can observe that the main ingredients of the proof of Theorem \ref{bergerwangnuevo} are Theorem \ref{desigualdad} and Poincar\'e's recurrence Theorem. Therefore, if we have another situation where an analogue of inequality \eqref{eqdesigualdad} holds, then we should obtain a result similar to Theorem \ref{bergerwangnuevo}. This is the case of cocycles of isometries of Gromov hyperbolic spaces. For definition and further properties of Gromov hyperbolicity see \cite{bri,papa,geodyn}.

As it was proved in \cite[Thm.~1.2]{yo}, if $M$ is a Gromov hyperbolic space with distance $d$, then there is a constant $C>0$ such that, for all $o\in M$ and $f,g$ isometries of $M$ we have
\small\begin{equation}
\small{d(fgo,o)\leq C+\max \left(d(fo,o)+d^{\infty}(g),d^{\infty}(f)+d(go,o),\frac{d(fo,o)+d(go,o)+d^{\infty}(fg)}{2}  \right)}, \notag
\end{equation}
\normalsize where $d^{\infty}(h)=\lim_{n\to \infty}\frac{d(h^{n}o,o)}{n}$ is the stable length.

In this context, given a probability space $(X,\mathcal{F},\mu)$ and a measure preserving map $T:X\rightarrow X$, a \textit{cocycle of isometries} of $M$ is a measurable map $A:X\rightarrow \Isom(M)$, where $\Isom(M)$ is the group of isometries of $M$, endowed with the Borel $\sigma-$algebra induced by the compact-open topology. We say that the cocycle $A$ is \textit{integrable} if the map $x\mapsto d(A(x)o,o)$ is integrable for some (and hence all) $o \in M$. In the same way as for linear cocycles, we define the family of maps $A^{n}:X\rightarrow \Isom(M)$. For references about cocycles of isometries, see e.g. \cite{goukar,karmar}. 

Following the same steps of the proof of Theorem \ref{bergerwangnuevo}, we can obtain the following:

\begin{prop}\label{bewagromov}Let $M$ be a Gromov hyperbolic space, $o\in M$, and let $T$ be a measure-preserving transformation of a probability space $(X,\mathcal{F},\mu)$. Also, let $A: X \rightarrow  \Isom(M)$ be an integrable cocycle of isometries of $M$. Then for $\mu$-almost all $x\in X$ and we have the following limits exist in $\R^{+}_{0}$ and are equal:
\begin{equation}
\limsup_{n\to \infty}{\frac{d^{\infty}(\ccl{A}{n}{x})}{n}}=\lim_{n \to \infty}{\frac{d(\ccl{A}{n}{x}o,o)}{n}}. \notag
\end{equation}
\end{prop}
\vspace{-2ex}A result similar to Proposition \ref{bewagromov} is far from being true if we do not assume a negative curvature condition on $M$.
\begin{ej}Let $X=\mathbb{S}^1$ and $\mu$ be the Lebesgue measure on $X$. If $T(z)=z^{2}$ is the doubling map on $X$, which preserves $\mu$, and $R_{a}(p)=p+a$ is the translation by $a\neq 0$ on $\R^{2}$, define a cocycle $A:\mathbb{S}^1 \rightarrow \Isom(\mathbb{R}^{2})$ by $A(z)p=T(z)R_{a}(z^{-1}p)$ for all $p\in \R^2$. Note that $\ccl{A}{n}{z}p=T^{n}(z)R^{n}_{a}(z^{-1}p)$ and hence the limit
$\lim_{n \to \infty}{\frac{d(\ccl{A}{n}{z}p,p)}{n}}$ exists and equals $|a|>0$ for all $z\in \mathbb{S}^1$ and $p\in \R^2$. On the other hand, if $z$ is not a periodic point for $T$, then $\ccl{A}{n}{z}$ is not a translation and hence has a fixed point. Thus we have that $d^{\infty}(\ccl{A}{n}{z})=0$ for all $n\in \N$ and all $z$ in the set of non periodic point of $T$, which is a full measure set with respect to $\mu$.
\end{ej}
\paragraph{Acknowledgment}
I am very grateful to J. Bochi for very interesting and valuable discussions throughout all this work. I also thank G. Urz\'ua for valuable discussions about Nullstellensatz, and the referee for the detailed report and the suggestions and corrections to the text. This article was supported by CONICYT scholarship 22172003, and partially supported by CONICYT PIA ACT172001.

\small{Eduardo Oreg\'on-Reyes (\texttt{eoregon@berkeley.edu})}\\
\small{Department of Mathematics}\\
\small{University of California at Berkeley}\\
\small{848 Evans Hall, Berkeley, CA 94720-3860, U.S.A.}


\begin{thebibliography}{26}

\footnotesize{

\bibitem{avibochi}
A.~Avila, J.~Bochi, A formula with some applications to the theory of Lyapunov
exponents. \textit{Israel Journal Math.} \textbf{131}, 125-–137, 2002.

\bibitem{bw}
M.~Berger, Y.~Wang, Bounded semigroups of matrices. \textit{Linear Algebra and its Applications},
\textbf{166}, 21--27, 1992.

\bibitem{bochilya}
J.~Bochi, Genericity of zero Lyapunov exponents. \textit{Ergodic Theory and Dynamical Systems,} \textbf{22}, 1667-–1696, 2002.

\bibitem{bociineq}
J.~Bochi, Inequalities for numerical invariants of sets of matrices. \textit{Linear Algebra and its
Applications}, \textbf{368}, 71–-81, 2003.

\bibitem{bro}
E.~Breuillard, A height gap theorem for finite subsets of $\mathrm{GL}_{d}(\overline{\mathbb{Q}})$ and nonamenable subgroups, Ann. of Math. \textbf{174}, 2, 1057-–1110, 2011.

\bibitem{bri}
M.~Bridson, A.~Haefliger, Metric spaces of non-positive curvature. Grundlehren der
Mathematischen Wissenschaften, \textbf{ 319},
Springer-Verlag, Berlin, 1999.

\bibitem{papa}
M.~Coornaert, T.~Delzant, and A.~Papadopoulos,
G{\'e}om{\'e}trie et th{\'e}orie des groupes : les groupes hyperboliques de Gromov,
Springer-Verlag, Berlin, 1990.

\bibitem{geodyn}
T.~Das, D.~S.~Simmons, and M.~Urba\'nski, Geometry and dynamics in Gromov hyperbolic metric spaces with an emphasis on non-proper settings. \textit{Mathematical Surveys and Monographs}. American
Mathematical Society, 2017.

\bibitem{elsner}
L.~Elsner, The generalized spectral-radius theorem: an analytic-geometric proof. \textit{Linear Algebra and
its Applications}, \textbf{220}, 151--159, 1995.

\bibitem{fulton}
M.~Fulton, Algebraic curves: An Introduction to algebraic geometry, 3rd
edition. \textit{Addison Wesley}, 2008.

\bibitem{fisicos}
I.~Goldhirsch, P.~Sulem, and S.~Orszag, Stability and Lyapunov stability of dynamical systems: a differential approach and a numerical method. \textit{Phys. D.} \textbf{27}, 3, 311-–337, 1987.

\bibitem{goukar}
S.~Gou\"ezel, A.~Karlsson, Subadditive and Multiplicative Ergodic Theorems. \textit{J. Eur. Math. Soc.}, to appear.

\bibitem{gur}
L.~Gurvits, Stability of discrete linear inclusion. \textit{Linear Algebra and its Applications}, \textbf{231}, 47–-85, 1995.

\bibitem{jun}
R.~Jungers, The joint spectral radius. \textit{Lecture Notes in Control and Information Sciences}, \textbf{385},
Springer-Verlag, Berlin, 2009.

\bibitem{karmar}
A.~Karlsson, G.~A.~Margulis, A multiplicative ergodic theorem and nonpositively curved spaces. \textit{Comm. Math. Phys.} \textbf{208}, 107–-123, 1999.

\bibitem{kozyakin} 
V.~Kozyakin, The Berger-Wang formula for the Markovian joint spectral radius. \textit{Linear Algebra and its Applications}, \textbf{448}, 315–-328, 2014.

\bibitem{klpasom}
T.~Krick, L.~M.~Pardo, and M.~Sombra, Sharp estimates for the arithmetic Nullstellensatz. \textit{Duke Math. J.} \textbf{109}, no. 3, 521–-598, 2001.

\bibitem{langalg}
S.~Lang, Algebra, third ed. \textit{Graduate Texts in Math}. \textbf{211}, Springer-Verlag, New York, 2002.

\bibitem{lorentz}
F.~Lorenz, Algebra. Volume II: Fields with structure, algebras, and advanced topics. \textit{Universitext}, Springer, 2008.

\bibitem{mcnaugh}
R.~McNaughton, Y.~Zalestein, The Burnside problem for semigroups.
\textit{Journal of Algebra}, \textbf{34}, 292--299, 1975.

\bibitem{morrispr}
I.~D.~Morris, An inequality for the matrix pressure function and applications. \textit{Advances in Mathematics}, \textbf{302}, 280–-308, 2016.

\bibitem{mathermorris}
I.~D.~Morris, Mather sets for sequences of matrices and applications to the study of joint spectral radii. \textit{
Proc. London Math. Soc.} \textbf{107}, 121–-150, 2013.

\bibitem{morrisbw}
I.~D.~Morris, The generalised Berger-Wang formula and the spectral radius of linear cocycles. \textit{Journal of Functional Analysis} \textbf{262}, 811--824, 2012.

\bibitem{yo}
E.~Oreg\'on-Reyes, Properties of sets of isometries of Gromov hyperbolic spaces. \textit{Groups Geom. Dyn.}, \textbf{12}, no. 3, 889--910, 2018.

\bibitem{raro}
H.~Radjavi, P.~Rosenthal, Simultaneous Triangularization. Springer-Verlag, New York, 2000.

\bibitem{rota}
G.~C.~Rota, G.~Strang, A note on the joint spectral radius. \textit{Indagatione Mathematicae}, \textbf{22},
379--381, 1960.

\bibitem{schr}
R.~Schneider, Convex bodies: the Brunn-Minkowski theory. Vol. 151 of \textit{Encyclopedia of Mathematics and its Applications, Cambridge University Press}, Cambridge, expanded ed., 2014.

\bibitem{som}
M.~Sombra, A sparse effective Nullstellensatz. \textit{Adv. in Appl. Math.} \textbf{22}, 271-–295, 1999. 

\bibitem{tsblo}
J.~N.~Tsitsiklis, V.~D.~Blondel, The Lyapunov exponent and joint spectral radius of pairs of matrices are hard —when not impossible— to compute and to approximate. \textit{Math.
Control Signals Systems}, \textbf{10}, 31–-40, 1997.
}




\end{thebibliography}
\end{document}